%% file: main.tex
\documentclass[a4paper,11pt]{amsart}
\input{Macros}
\usepackage[english]{babel}
\usepackage{tabu}
\usepackage{mathrsfs}
\usepackage{bbm}

\newcommand{\fm}{\mathfrak{m}}

\title{Local stability of Frobenius and Artin--Schreier base-change: a comparison}
\author{Quentin Posva}
\date{}

\address{Institut de Mathématiques, Université de Neuchâtel, Rue Émile-Argand 11, 2000 Neuchâtel, Switzerland
\\
ORCID: 0000-0003-2184-3763} 
\email{quentin.posva@unine.ch}

\begin{document}

\maketitle

\begin{quote}
\textsc{Abstract.} Let $(X,\Delta)\to C$ be a family of pairs over a smooth curve in positive characteristic $p>0$.
We study the discrepancies of divisors over the base-change $(X,\Delta)\to C$ along Artin--Schreier and Frobenius covers of $C$, through ramification data for Artin--Schreier and degree $p$ purely inseparable extensions of DVRs. As an application, we recover a result of Hu and Zong about permanence of local stability for such base-change.

\medskip\noindent
\textbf{MSC codes}: 14J17, 14E30, 14E22.
\end{quote}

\setcounter{tocdepth}{1}
\tableofcontents

\input{Introduction}

\input{Preliminaries}

\input{AS_ramification}

\input{Infinitesimal_ramification}

\input{Application_to_lc_families}

\bibliographystyle{alpha}
\bibliography{Bibliography}

\end{document}

%% file: Introduction.tex
\section{Introduction}
The condition of \emph{local stability} was originally introduced by Koll\'{a}r and Shepherd--Barron in their studies of families of surfaces \cite{Kollar_Shepherd_Barron_3folds_and_deformations_of_surfaces_singularities}, and latter turned out to be an important technical condition for the definition and construction of compact moduli spaces of varieties of general type in characteristic zero 
\cite{Kollar_Families_of_varieties_of_general_type}. It is expected that local stability is also a condition of interest in positive characteristic, but so far very few desirable properties are known (see \cite{Arvidsson_Bernasconi_Patakfalvi_Properness_moduli_surfaces} for the case of families of surfaces).

A basic statement that is so far unsettled in positive characteristic is whether local stability is preserved by base-change, when the base is a smooth curve. The question was recently investigated in \cite{Hu_Zong} and \cite{Benozzo-Posva}. In the first paper, Hu and Zong argue that the case of Artin--Schreier base-change can be reduced to Frobenius base-change. In the second one, Benozzo and the author use Hu--Zong's result as a starting point to reduce the case of arbitrary base-change to Frobenius base-change, before studying the latter.

I must confess that the technicalities of \cite{Hu_Zong} evade me, and this note grew up as an effort to clarify and streamline Hu--Zong's argument. In the end we recover the same statement:
\begin{theorem}[\autoref{thm:main_application}]\label{main_thm_intro}
Let $X\to C$ be a strongly locally stable family of varieties, where $C$ is a smooth curve over an algebraically closed field $k$ of positive characteristic. Let $C'\to C$ be an Artin--Schreier cover by a smooth curve. Then $X\times_CC'\to C'$ is strongly locally stable, granted $X\times_{C,F^n}C\to C$ is strongly locally stable for all $n\geq 1$, where $F^n$ is the $n$-th $k$-linear Frobenius morphism of $C$.
\end{theorem}

For the definition of strong local stability, we refer to \cite[\S 2]{Benozzo-Posva}: it is a more restrictive version of local stability, designed to avoid some possible pathologies on the base-change that may occur in positive characteristics (\footnote{
    See the beginning of \autoref{section:application} for more details. In particular, strong stability of $X\to C$ ensures that the fiber products $X\times_CC'$ and $X\times_{C,F^n}C$ are normal with $\bQ$-Cartier canonical divisors, which is tacitely assumed in \cite{Hu_Zong}.
}).

Here is a summary of our approach. Given a birational morphism $Y\to X$, we let $Y'=(Y\times_CC')^\nu$ and consider a prime divisor $D\subset Y'$ dominating $E\subset Y$. The non-trivial case to study is when: $Y'\to Y$ is ramified at the generic point of $D$ and $E$ maps to a closed point $\mathbf{0}\in C$. In this case, formulas for the discrepancy of $D$ in terms of the discrepancy of $E$ and ramification data of $\sO_{Y,E}\hookrightarrow \sO_{Y',D}$ are given in \autoref{prop:discrep_AS_divisor}.
When $E$ is \emph{tame}, i.e.\ when $\coeff_EY_\mathbf{0}$ is prime to $p$ \cite[\S 3]{Benozzo_CBF}, these formulas implies immediately that $a(D;X',X'_{\mathbf{0}'})\geq -1$ (\autoref{cor:AS_discrep_tame_div}). Difficulties appear when $E$ is \emph{wild}, i.e.\ when $\coeff_EY_\mathbf{0}$ is divisible by $p$: the input from the Frobenius base-change is useful to control these divisors.

Let $E^{(n)}$ denote the unique prime divisor on the normalized Frobenius base-change $Y^{(n)}=(Y\times_{C,F^n}C)^\nu$ which dominates $E$. If $E^{(n)}$ is tame for $n\gg 1$, then it follows from the tame case and from Frobenius descent of local stability \cite[\S 6.1]{Benozzo-Posva} that the discrepancy of $D$ is at least $-1$.

However, contrary to the initial hope of the author, it is \emph{not} true that every wild divisor eventually becomes tame through Frobenius base-change. For every characteristic $p>2$, a counterexample is given in \autoref{example:stubborn_divisor}. The case $p=3$ was communicated to the author by Akiyoshi Sannai, who found it using an AI referee tool, \textsc{AI Math Referee}, that he is currently developing. 

Suppose that $E$ is not eventually tame. By Frobenius descent, we are free to replace $E$ by any $E^{(n)}$ to analyze the situation. In particular, we can assume that the multiplicity of $E$ in its fiber is unchanged by taking Frobenius base-change, in which case we say that $E$ is \emph{stubborn} (\autoref{def:stubborn_div}). Using some tools from the theory of $n$-foliations (see \cite{Grabowski_Power_towers} or the upcoming \cite{foliations_in_progress}), we show that the discrepancy of a stubborn divisor $E$ is bigger or equal to its wild coefficient (\autoref{cor:discrep_bigger_than_wild_coeff}; see \cite[\S 4]{Benozzo-Posva} for the definition of wild coefficients).
Combining this inequality with the formulas found in \autoref{prop:discrep_AS_divisor} and some simple relations between the wild coefficient and the quantities measuring the ramification of $D$ and $E^{(1)}$ over $E$, we are able to bound from below the discrepancy of $D$.

A comparison between Hu--Zong's approach and ours is given in \autoref{rmk:comp_with_HZ}. While the first steps are similar, the two approaches to the non-eventually-tame divisors are not easy to compare. Hu and Zong study the equation whose normalization defines $Y'$, through a wealth of algebraic manipulations and partial normalizations via blow-ups. 

While our approach leads to the same result, we believe it is more transparent. In \autoref{section:AS_ext} we present the theory of Artin--Schreier extensions solely in terms of the function $\lambda(-)$
(\footnote{
    The notation $\lambda(-)$ is ours, but the quantity it describes is classical in ramification theory: see for example \cite[\S 2.1]{Xiao_Zhukov_Ramification}.
}): we develop more than what is strictly needed for the applications in \autoref{section:application}, as it may be of independent interest.
On the Frobenius side, the main insight is the introduction of the function $\widetilde{\lambda}(-)$ in \autoref{section:pur_insep_ramification}.
As an additional bonus, our method allows us to treat the case of families of pairs $(X,\Delta)\to C$.

\subsection{Acknowledgments}
I am thankful to Marta Benozzo for several remarks and discussions, and to Thor Wittich and to Akiyoshi Sannai for signaling issues in previous versions of this paper.
The author is supported by the Ambizione grant num.\ 
PZ00-2-233513 from the Swiss National Science Foundation, and thanks the Université de Neuchâtel for its hospitality.

\subsection{AI disclosure}
\autoref{example:stubborn_divisor} was essentially found by Akiyoshi Sannai using AI. The rest of the paper was written without any assistance from AI.

%% file: Preliminaries.tex
\section{Preliminaries}
Throughout this note, we work over an algebraically closed field $k$ of characteristic $p>0$. In \autoref{section:application} we follow the standard terminology of birational geometry, as in \cite{Kollar_Singularities_of_the_minimal_model_program}. We refer to \cite{Posva_Sing_qt_foliations} for the few notions related to $1$-foliations that we will use.

\medskip
We recall some basic facts about strict henselization and completions.
Let $A$ be an excellent local ring over $k$, with maximal ideal $\fm$ and residue field $k_A=A/\fm_A$. We denote by $A^{\mathrm{csh}}$ the completion of the strict henselization of $A$ (here $\mathrm{csh}$ stands for \emph{completed strict henselization}). It is an excellent, local, strictly henselian ring of the same dimension as $A$, with maximal ideal generated by $\fm$ and with residue field a separable closure $k_A^\mathrm{sep}$ of $k_A$ \cite[07QL]{Stacks_Project}. 
The extension $A\hookrightarrow A^\mathrm{sh}$ is formally \'{e}tale, and $A^\mathrm{sh}\hookrightarrow A^\mathrm{csh}$ is regular. 

Suppose that $A$ is integral, let $L$ be a finite extension of $\Frac(A)$, and suppose that the normalization $B$ of $A$ in $L$ is local. Then $B^\mathrm{csh}\cong B\otimes_AA^\mathrm{csh}$ by \cite[18.8.10]{EGA_IV.4} and \cite[7.8.3.(vii)]{EGA_IV.3}.

Recall also that if $A$ is essentially of finite type over $k$, then the dual of $\Omega_A^1\otimes_AA^\mathrm{csh}$ is equal to $\Der^\mathrm{cont}_k(A^\mathrm{csh})$ \cite[Lemma 2.3.1]{Posva_Resolution_of_1_foliations}. When $A$ is integral Gorenstein, this can be used to reduce computations involving $\omega_A$ to $A^\mathrm{csh}$.

\medskip 
We also record the following ancillary result:
\begin{lemma}\label{lemma:affine_version_of_Stein}
Let $C$ be an integral $k$-scheme, and let $f\colon X\to C$ be a flat morphism. Assume that $X$ is integral and that the generic fiber of $f$ is geometrically integral. Then $k(C)$ is algebraically closed in $k(X)$.
\end{lemma}
\begin{proof}
Assume that $k(C)$ is not algebraically closed in $k(X)$: then there exists a non-empty open set $U\subset X$ and a section $s\in \Gamma(U,\sO_X)$ such that the image of $s$ through $\Gamma(U,\sO_X)\to k(X)$ does not belong to $K=k(C)$, but such that $K\hookrightarrow K(s)=L$ is a finite and non-trivial field extension. Thus, if $\overline{K}$ is an algebraic closure of $K$, the spectrum of $L'=L\otimes_{K}\overline{K}$ is either disconnected or non-reduced. 

By definition of $L$, the morphism $f|_{U_{K}}$ factors as
        $$U_K \overset{\varphi}{\longrightarrow}
        \Spec(L) \longrightarrow \Spec(K)$$
where $\varphi$ is flat.
Base-changing along $\Spec(\overline{K})\to \Spec(K)$, we obtain that $f|_{U_{\overline{K}}}$ factors as
        $$U_{\overline{K}} \overset{\overline{\varphi}}{\longrightarrow}
        \Spec(L') \longrightarrow \Spec(\overline{K})$$
where $\overline{\varphi}$ is flat and surjective. 

If $L'$ is non-reduced, so is $U_{\overline{K}}$ \cite[06QM]{Stacks_Project}, and as $U_{\overline{K}}$ is open inside $X_{\overline{K}}$ we obtain a contradiction with the hypothesis.
If $\Spec(L')$ is disconnected, then $U_{\overline{K}}$ is also disconnected; this implies that $X_{\overline{K}}$ is not irreducible, which is also a contradiction with the hypothesis. Thus the proof is complete.
\end{proof}

%% file: AS_ramification.tex
\section{Ramification of Artin--Schreier extensions}\label{section:AS_ext}
We recall the basic theory of Artin--Schreier extensions for fields and DVRs. This is mostly well-known: see for example \cite[Propositions 3.7.7 and 3.7.8]{Stichtenoth_Algebraic_fct_fields_and_codes} and \cite[\S 2.1]{Xiao_Zhukov_Ramification}. We also extend some results of \cite[\S 5]{Lorenzini-Schroeer} from dimension one to codimension one in \autoref{section:effective_actions}.

\subsection{Artin--Schreier extensions of fields}
The story is quite simple: let $K$ be a field extension of $k$. We introduce the additive function
        $$\wp\colon K\longrightarrow K, \quad 
        f\mapsto f^p-f,$$
which is only $\bF_p$-linear. The Artin--Schreier theorem \cite[09I7]{Stacks_Project} states that the Galois extensions of degree $p$ of $K$ are precisely those of the form
        $$K_{\wp,f}=K[X]/(X^p-X-f), \quad f\notin \wp(K).$$
The Galois $\bZ/p$-action over $K$ is generated by $X\mapsto X+1$. The following is well-known:
\begin{lemma}\label{lemma:equiv_AS_ext}
Two extensions $K_{\wp,f}$ and $K_{\wp,g}$ are isomorphic as $K$-algebras if and only if $f=\zeta g+\wp(h)$ with $\zeta\in \mathbb{F}_p^\times$, via
        $$K[X]/(X^p-X-f)\cong K[Y]/(Y^p-Y-g), \quad 
        X\mapsto \zeta Y+h.$$  
\end{lemma}

\subsection{Artin--Schreier extensions of DVRs}\label{section:AS_ext_DVR}
Let $K$ be as above, and let $\upsilon$ be a normalized discrete valuation of $K$ with associated excellent DVR $A\subset K$. We let $\fm_A$ be the maximal ideal of $A$, and $k_A=A/\fm_A$ be the residue field. We do not assume $k_A$ to be perfect.

Consider an Artin--Schreier extension $K\hookrightarrow K_{\wp,f}$ and let $B\subset K_{\wp,f}$ be the normalization of $A$ in $K_{\wp,f}$: it is a Dedekind domain, which is a finite $A$-module. The Galois $\bZ/p$-action on $K_{\wp,f}$ restricts to a $\bZ/p$-action on $B$, with sub-ring of invariants $A$ (see for example the proof of \cite[Lemma 6.2.1]{Benozzo-Posva}). Since $\bZ/p$ is a simple group, it follows from \cite[\S X]{Raynaud_Anneaux_locaux_henseliens} that we have the following dichotomy:
    \begin{enumerate}
        \item either the $\bZ/p$-action on $B$ is free and $A\hookrightarrow B$ is \'{e}tale; or,
        \item $B$ is local and the $\bZ/p$-action on its residue field is trivial.
    \end{enumerate}
In the first case, it may happen that $B$ is local; but after base-changing to the strict henselization $A^\mathrm{sh}$ of $A$, we obtain $B\otimes_A A^\mathrm{sh}\cong (A^\mathrm{sh})^{\oplus p}$ as $A^\mathrm{sh}$-algebras. In the second case, we say that $B$ is \emph{ramified} over $A$.

\emph{From now on we assume that $B$ is local and ramified over $A$.} In particular $B$ is a DVR, with maximal ideal $\fm_B$ and residue field $k_B=B/\fm_B$. We denote by $\upsilon'$ the normalized valuation of $B$.

\begin{lemma}\label{lemma:AS_parameter_not_regular}
In the above situation, we have $f\notin A$.
\end{lemma}
\begin{proof}
It is enough to prove the claim after extending along the strict henselization of $A$. So we may assume that $A$ is strictly henselian, and thus $k_A$ is separably closed. If $f\in A$ and $\bar{f}\in k_A$ denotes its residue, the separable polynomial $X^p-X-\bar{f}\in k_A[X]$ has a root $\alpha\in k_A$. Since $\alpha+i$ is also a root for any $i\in \bF_p$, the polynomial splits completely in $k_A[X]$. It follows from Hensel's lemma that $X^p-X-f\in A[X]$ also splits into linear factors, and thus $K[X]/(X^p-X-f)\cong K^{\oplus p}$. So $B\cong A^{\oplus p}$ is \'{e}tale over $A$ if $f\in A$. As we assumed $B$ to not be \'{e}tale over $A$, we obtain $f\notin A$.
\end{proof}

\noindent 
It follows from ramification theory that in the ramified case we have the following further dichotomy \cite[\S 2.1]{Xiao_Zhukov_Ramification}:
    \begin{enumerate}
        \item[$(\dagger)$] \emph{Wild ramification}: the extension $k_A\hookrightarrow k_B$ is trivial;
        \item[$(\ddagger)$] \emph{Ferocious ramification}: the extension $k_A\hookrightarrow k_B$ is purely inseparable of degree $p$.
    \end{enumerate}
We introduce the following integers: the \emph{ramification index}
        $$e=\frac{p}{[k_B:k_A]}=
        \begin{cases}
            1 & \text{if the ramification is ferocious,}\\
            p & \text{if the ramifiction is wild,}
        \end{cases}$$
and the \emph{ramification break} (\footnote{
    For simplicity we deviate from the standard convention by $1$, cf.\ \cite[IV, \S 1]{Serre_Corps_locaux}.
}) 
        \begin{equation}\label{def:ramification_break}
        r=\max\left\{s\geq 1\mid 
        \bZ/p \text{ acts trivially on }B/\fm_B^{s}\right\}.
        \end{equation}
Denote by $q\colon X=\Spec(B)\to \Spec(A)=Y$ the induced morphism on spectra. Let $\mathbf{x}_A=V(\fm_A)\subset Y$ and $\mathbf{x}_B=V(\fm_B)\subset X$ be the closed points, which we regard as effective Cartier divisors. Then:
        \begin{equation}\label{eqn:ramification_AS_ext}
        q^*\mathbf{x}_A=e\cdot \mathbf{x}_B \quad 
        \text{and} \quad 
        K_X=q^*K_Y+(p-1)r\cdot \mathbf{x}_B,
        \end{equation}
see \cite[\S 3]{Totaro_Terminal_3folds_not_CM}. In particular, if $\pi$ is an uniformizer of $A$ then $\upsilon'(\pi)=e$.

To compute the ramification type and the integers $e$ and $r$, we are allowed to replace
$A$ by its strict henselization $A^\mathrm{sh}$ or by the completion $A^\mathrm{csh}$ of the latter. We will use this remark several times below.

\medskip
Let us now explain how to compute the integers $e$ and $r$ from the data of $f\in K$. We introduce the following quantity:

\begin{definition}\label{def:AS_lambda}
Let $K^\mathrm{csh}=\Frac(A^\mathrm{csh})$ and $\upsilon^\mathrm{csh}$ denote the normalized discrete valuation of $A^\mathrm{csh}$.
For $g\in K^\mathrm{sh}$, we let
        $$\lambda_{K,\upsilon}(g)=\max\{ \upsilon^\mathrm{csh}(g+\wp(x))\mid 
        x\in K^\mathrm{csh}\}\in \bZ\cup \{\infty\}.$$
If the discrete valued field $(K,\upsilon)$ is clear from the context, we omit it from the notation; or, alternatively, we indicate the ring of integers, \emph{viz.}\ $\lambda_{K,\upsilon}(-)=\lambda_{A}(-)$.
\end{definition}

In practice, we can compute it using the following method:

\begin{algorithm}\label{algo:compute_lambda}
Fix a uniformizer $\pi$ of $A$ and fix a field of representatives $k_A\subset A$, so that
$A^\mathrm{csh}\cong k_A^\mathrm{sep}\llbracket \pi\rrbracket$ by Cohen's structure theorem. Given $g\in K^\mathrm{csh}$, we write $g=\sum_ig_i\pi^i$ its Laurent series expansion. 

Given $g\in K^\mathrm{csh}$, consider the following algorithm run in $K^\mathrm{csh}$:
    \begin{enumerate}
        \item[(1)] Consider the following assertion:
                $$\upsilon^\mathrm{csh}(g)=n\in p\bZ
                \quad \text{and} \quad g_{n}\in (k_A^\mathrm{sep})^p.$$
        If this assertion about $g$ if false: return the integer $\upsilon^\mathrm{csh}(g)$. If this assertion about $g$ is true: go to the second step.

        \item[(2)] If $\upsilon^\mathrm{csh}(g)=pm$ with $m\in \bZ$ and $g_{pm}=u^p\in (k_A^\mathrm{sep})^p$, replace $g$ by
        $g-\wp(u\pi^{m})$ and go back to Step 1.
    \end{enumerate}
If $\upsilon(g)<0$ and $g\notin \wp(K^\mathrm{csh})$, this algorithm terminates after finitely many steps and its output is $\lambda_{K,\upsilon}(g)\leq 0$.
\end{algorithm}
\begin{proof}
If $\upsilon^\mathrm{csh}(h)>0$ then the reduction of the polynomial $X^p-X-h$ in $k_A^\mathrm{sep}[X]$ splits into linear factors, and by Hensel's lemma so does $X^p-X-h$. This shows that $\fm_{A^\mathrm{csh}}\subset \wp(K^\mathrm{csh})$. Moreover, since $k_A^\mathrm{sep}$ is separably closed, its Artin--Schreier map and surjective and so $k_A^\mathrm{sep}=\wp(k_A^\mathrm{sep})\subset \wp(K^\mathrm{csh})$. So, under the given hypothesis on $g$, to compute $\lambda(g)$ we may assume that $g=\sum_{i=-n}^{-1}g_i\pi^i$ with $g_i\in k_A^\mathrm{sep}$. Then the algorithm clearly stops after finitely many steps, and outputs a negative integer.
\end{proof}

\begin{lemma}\label{lemma:lambda_well_defined}
In our situation, $\lambda_{K,\upsilon}(f)< 0$ and its depends only on $K_{\wp,f}$.
\end{lemma}
\begin{proof}
Thanks to \autoref{lemma:AS_parameter_not_regular} and to \autoref{algo:compute_lambda},
to prove $\lambda_{K,\upsilon}(f)< 0$ it suffices to show that $f\notin \wp(K^\mathrm{csh})$. Since $B$ is local and finite over $A$, by \cite[18.8.10]{EGA_IV.4} we see that $A^\mathrm{sh}\otimes_AB$ is local and regular. It is the normalization of $A^\mathrm{sh}$ in $(K^\mathrm{sh})_{\wp,f}$, where $K^\mathrm{sh}=\Frac(A^\mathrm{sh})$. So the completion of $A^\mathrm{sh}\otimes_AB$ is local, and by \cite[7.8.3.(vii)]{EGA_IV.3} it is isomorphic to the normalization of $A^\mathrm{csh}$ $(K^\mathrm{csh})_{\wp,f}$. This show that $f\notin \wp(K^\mathrm{csh})$ as desired.

Suppose that $K_{\wp,f}\cong K_{\wp,g}$ as extensions of $K$. Then $g=\zeta f+\wp(h)$ for some $\zeta\in\mathbb{F}_p^\times$, according to \autoref{lemma:equiv_AS_ext}. So 
$\lambda(f)=\lambda(\zeta g)=\lambda(g)$, where the second equality follows from the observation
$\zeta g+\wp(\zeta h)=\zeta(g+\wp(h))$.
\end{proof}


We claim that $\lambda(f)$ determines the ramification data of $A\hookrightarrow B$ introduced above, i.e.\ wild/ferocious ramification, ramification index and ramification breaks.
It is well-known that
        \begin{equation}\label{eqn:lambda_and_ram_type}
        p\text{ divides }\lambda(f)
        \quad \Longleftrightarrow \quad 
        B\text{ is ferociously ramified over }A,
        \end{equation}
see \cite[\S 2.1]{Xiao_Zhukov_Ramification} or \cite[09F0]{Stacks_Project}. So we see already that $\lambda(f)$ determines the ramification type and the ramification index $e$. Let us show that it also determines the ramification break $r$.

\begin{proposition}\label{prop:ramif_break_wild_case}
Assume that $B$ is wildly ramified over $A$. Then $r=-\lambda(f)+1$.
\end{proposition}
\begin{proof}
We may assume that $A$ is strictly henselian and that $\upsilon(f)=\lambda(f)=-s$.
Let $\pi\in A$ be a uniformizer and let $x\in \Frac(B)=K[X]/(X^p-X-f)$ be an element satisfying $x^p-x=f$; their valuation on $B$ is
$\upsilon'(\pi)=p$ and $\upsilon'(x)=-s$. Let $a,b\geq 1$ be positive integers satisfying $-as+bp=1$. Then
$\upsilon'(x^a\pi^b)=1$ so $x^a\pi^b$ is a uniformizer of $B$. A generator of the $\bZ/p$-action on $B$ sends $x^a\pi^b$ to $(x+1)^a\pi^b$; moreover $B$ is generated as $A$-module by $\{(x^a\pi^b)^j\mid 0\leq j\leq p-1\}$
\cite[III, \S 6, Lemme 3 on p.66]{Serre_Corps_locaux}. So $r$ is the biggest integer satisfying
        $$(x+1)^a\pi^b-x^a\pi^b=
        \sum_{i=1}^a\binom{a}{i}x^{a-i}\pi^b
        \in \fm_B^r.$$
We have $\upsilon'(x^{a-i}\pi^b)=-s(a-i)+pb$. Taking in account that $a$ is positive and prime to $p$, so that the coefficient of the term $i=1$ is non-zero, we obtain that 
$r=-s(a-1)+pb=1+s$ as claimed.
\end{proof}

\begin{proposition}\label{prop:ramif_break_feroce_case}
Assume that $B$ is ferociously ramified over $A$. Then $r=-\lambda(f)/p$.
\end{proposition}
\begin{proof}
We may assume that $A$ is strictly henselian and that $\lambda(f)=\upsilon(f)$. Choose an uniformizer $\pi$ of $A$, and let us write $f=u\pi^{-pn}$ where $n\geq 1$ is an integer and $u\in A^\times$ is a unit. We claim that the residue $\bar{u}\in k_A$ does not belong to $k_A^p$. For the sake of contradiction, assume $\bar{u}\in k_A^p$; then we can write 
$u=v^p+x\pi $ for some $x\in A$ and unit $v\in A^\times$. Hence
        $$f-\wp(v\pi^{-n})=
        v^p\pi^{-pn}+x \pi^{-pn+1}
        -v^p\pi^{-pn}+v\pi^{-n}
       =x\pi^{-pn+1}+v \pi^{-n}$$
which shows that $\upsilon(f)<\upsilon(f+\wp(v\pi^{-n}))$, contradicting our choice of $f$. So $\bar{u}\notin k_A^p$.

Setting $Y=\pi^nX$ yields an isomorphism of $K$-algebras
        \begin{equation}\label{eqn:change_presentation_field}
        K_{\wp,f}=K[X]/(X^p-X-f)\cong 
        K[Y]/(Y^p-\pi^{n(p-1)}Y-u),
        \end{equation}
and the Galois $\bZ/p$-action on the field on the right-hand side is generated by $Y\mapsto Y+\pi^n$.
Hence $A'=A[Y]/(Y^p-\pi^{n(p-1)}Y-u)$ is an intermediate extension of $A\hookrightarrow B$. In particular there is an element $y\in B$ satisfying $y^p-\pi^{n(p-1)}y-u=0$. In $k_B$ we have $\bar{y}^p-\bar{u}=0$; as $\bar{u}\notin k_A^p$ we obtain that $\bar{y}\notin k_A$. Therefore $B$ is generated as $A$-module by $\{y^j\mid 0\leq j\leq p-1\}$ \cite[III, \S 6, Lemme 3 on p.66]{Serre_Corps_locaux}. As a generator of the $\bZ/p$-action sends $y$ to $y+\pi^n$, and $\upsilon'(\pi)=1$, we deduce that $r=n=-\lambda(f)/p$. 
\end{proof} 

\begin{remark}
It follows easily from the two previous propositions that:
The positive integers that can occur as ramification breaks of wildly ramified actions are precisely those that do not belong to $1+p\bZ$, while every positive integer can occur as the ramification break of a ferouciously ramified action. In particular, the ramification break alone does not always determine the ramification type.
\end{remark}

\subsubsection{Effective actions and torsor structures}\label{section:effective_actions}
To conclude this section, we consider the extension $A\hookrightarrow B$ from the perspective of group scheme actions (\footnote{These results will not be used in the rest of the article.}).

The Galois $\bZ/p$-action on $\Frac(B)$ extends to an $A$-linear linear $\bZ/p$-action on $B$. We have seen that $B$ is \'{e}tale over $A$ if and only this action is free, which is equivalent to $B$ being a $\bZ/p$-torsor over $A$. So, for the rest of this sub-section, \emph{we assume that the extension $A\hookrightarrow B$ is ramified.}

Let $X=\Spec(B)$ and $Y=\Spec(A)$.
The Galois action yields a morphism of group functors
        \begin{equation}\label{eqn:generator_effective_model}
        (\bZ/p)_{Y}\to \underline{\Aut}_{X/Y}.
        \end{equation}
Since $X\to Y$ is finite and flat, $\underline{\Aut}_{X/Y}$ is representable by a disjoint union of affine finite type group schemes over $Y$---\cite[Part 2, \S 5.6.2]{FGA_explained} and \cite[Lemma 4.1]{Lorenzini-Schroeer}. We denote the schematic image of \autoref{eqn:generator_effective_model} by $\mathscr{G}$.

\begin{proposition}
$\mathscr{G}$ is a flat finite commutative $Y$-group scheme of length $p$.
\end{proposition}
\begin{proof}
The fact that $G(\wp,f)$ is a flat finite $C$-group scheme is a well-known result, essentially going back to Raynaud: see \cite[\S 4.3]{Romagny_Effective_models}, \cite[\S 2.1]{Abramovich_Raynaud_grp_scheme} or \cite[\S 4]{Lorenzini-Schroeer}. We have  $\mathscr{G}\times_Y \Spec(K)=(\bZ/p)_{\Spec(K)}$ because a finite union of $K$-points on a finite type $K$-scheme is a closed sub-scheme. 
By flatness we deduce that $\mathscr{G}$ has length $p$ over $C$. Commutativity now follows from \cite[Theorem 1 on p.5]{Tate-Oort}.
\end{proof}

Following \cite[Definition 4.3.3]{Romagny_Effective_models}, the $Y$-group scheme $\mathscr{G}$ is called the \emph{effective model} of the action of $\bZ/p$ on $X$. A remarkable property of this model is the following: since the structural morphism $\mathscr{G}\to \underline{\Aut}_{X/Y}$ is a closed immersion, its kernel sheaf is trivial.

Finite flat group schemes of length $p$ are classified by Tate and Oort in \cite{Tate-Oort}, to which the reader is referred for the definition of $G^\mathscr{L}_{a,b}$ (see \emph{op.cit.}, p.14). Their theory applies in particular to $\mathscr{G}$. It is possible to find the isomorphism class of $\mathscr{G}$ without knowing an explicit presentation of $B$ as $A$-algebra: the following elegant argument is essentially due to Lorenzini and Schröer. 

\begin{proposition}[cf.\ {\cite[Proposition 5.7]{Lorenzini-Schroeer}}]\label{prop:effective_action}
Let $\pi$ be a uniformizer of $A$. Then
        $$\mathscr{G}\cong G^{\sO_Y}_{\pi^{l(p-1)},0}
        \quad \text{with} \quad 
        l=\begin{cases}
        \lfloor (-\lambda(f)+1)/p\rfloor & \text{if the ramification is wild,}\\
        -\lambda(f)/p & \text{if the ramification is ferocious.}
        \end{cases}$$
\end{proposition}
\begin{proof}
We have seen that the generic fiber of $\mathscr{G}\to Y$ is the constant group associated to $\bZ/p$. So Tate--Oort theory implies that $\mathscr{G}\cong G^{\sO_Y}_{\pi^{l(p-1)},0}$ for some integer $l\geq 0$, see \cite[Proposition 5.3]{Lorenzini-Schroeer}. Hence
        $$\mathscr{G}\cong \Spec A[x]/(x^p-\pi^{l(p-1)}x)$$
where the ideal defining the neutral section $Y\to \mathscr{G}$ is $(x)$. The ideal $(x-\pi^l)$ defines an $Y$-point $\xi$ of $\mathscr{G}$. The structural morphism $\phi\colon (\bZ/p)_Y\to \mathscr{G}$ is an isomorphism over $K$: so $\xi_K$ extends uniquely to an $Y$-point $\sigma$ of the constant group $(\bZ/p)_Y$ satisfying $\phi(\sigma)=\xi$. 

For any integer $n\geq 1$, we write $\xi_n=\xi\otimes_A A/(\pi^n)$ and $\sigma_n=\sigma\otimes_A A/(\pi^n)$: these are $A/(\pi^n)$-linear automorphisms of $B/(\pi^nB)$. 
Since the structual $j\colon \mathscr{G}\to \underline{\mathrm{Aut}}_{X/Y}$ is a closed embedding, 
we see that $\xi_n=\id_{B/(\pi^nB)}$ if and only if $n\leq l$. On the other hand, by definition of the ramification break $r$, the element $\sigma_n$ is mapped to the identity by $j\circ \phi$ if and only if $\pi^nB\supseteq \mathfrak{m}_B^r$, which is equivalent to $n\upsilon'(\pi)\leq r$. Therefore we obtain
        $$n\upsilon'(\pi)\leq r
        \  \Longleftrightarrow \ 
        n\leq l,$$
which implies $l= \lfloor r/\upsilon'(\pi)\rfloor$. We obtain the desired expression thanks to \autoref{prop:ramif_break_wild_case} and \autoref{prop:ramif_break_feroce_case}.
\end{proof}

The following is a generalization of \cite[Theorem 5.1]{Lorenzini-Schroeer}:

\begin{proposition}\label{prop:torsor_AS_case}
The $\mathscr{G}$-action on $X$ makes $X\to Y$ into a $\mathscr{G}$-torsor if and only if:
    \begin{itemize}
        \item the ramification is wild and $p$ divides $-\lambda(f)+1$; or,
        \item the ramification is ferocious.
    \end{itemize}
\end{proposition}
\begin{proof}
Let $l\geq 1$ be as in \autoref{prop:effective_action}.
According to \cite[Lemma 4.14 and its proof]{Lorenzini-Schroeer}, $\mathscr{G}$-torsors over $Y$ are isomorphic (as torsors) to
        $$\Spec A[T]/(T^p-\pi^{l(p-1)}T-s)\to \Spec(A)$$
for some choice of uniformizer $\pi$ of A and some element $s\in A$, with $\mathscr{G}$-action induced by $T\mapsto T+\pi^l$. If the total space of the torsor is integral, the induced extension of function fields is $K\hookrightarrow K_{\wp,s\pi^{-lp}}$.

Suppose that the total space of such a torsor is isomorphic to $B$ over $A$. Then $\lambda(f)=\lambda(s\pi^{-lp})$, which we denote by $\lambda$. Moreover, the hypersurface above is formally smooth over $k$, and by the Jacobian criterion we obtain that $s\in A$ is a differential coordinate, i.e.\ the differential $ds\in \Omega^1_{A/k}$ generates a direct summand. The reader will check that this implies
        $$\lambda(s\pi^{-lp})=
        \upsilon(s\pi^{-lp})=
        \begin{cases}
        -lp & \text{if }s\in A^\times,\\
        -lp+1 & \text{otherwise}.
        \end{cases}$$
We also obtain that $B$ is ferociously ramified over $A$ if and only if $s\in A^\times$. In the ferocious case, the equality $\lambda=-lp$ is nothing new in view of \autoref{prop:effective_action}. In the wild case, we get $l=(-\lambda+1)/p$: so $-\lambda+1$ must be divisible by $p$.

We prove the converse. We begin with the wild case: suppose that $-\lambda(f)+1$ is divisible by $p$. Then \autoref{prop:effective_action} implies that 
$lp=-\lambda(f)+1$. After going to the strict henselization we may assume that $f=\pi^{\lambda(f)}$ for some uniformizer $\pi$ of $A^\mathrm{sh}$, so that
        $$B^\mathrm{sh}\text{ is the normalization of }
        A^\mathrm{sh}[T]/(T^p-\pi^{l(p-1)}T-\pi);$$
but it is clear that the right-hand side is already regular, hence equal to $B^\mathrm{sh}$. The $\mathscr{G}\otimes_AA^\mathrm{sh}$-action is induced by $T\mapsto T+\pi^l$. By the first paragraph, we deduce that $B^\mathrm{sh}$ is a $\mathscr{G}\otimes_AA^\mathrm{sh}$-torsor above $A^\mathrm{sh}$: by descent, it follows that $B$ is a $\mathscr{G}$-torsor above $A$.

It remains to consider the ferocious case: then $\lambda(f)=-lp$ according to \autoref{prop:effective_action}. We may assume that $f=u\pi^{-lp}$ for some $u\in A^\times$, and \autoref{algo:compute_lambda} shows that $u$ is a differential coordinate of $A$. This implies that $A[T]/(T^p-\pi^{l(p-1)}T-u)$ is already regular, with $\mathscr{G}$-action induced by $T\mapsto T+\pi^l$. So $B$ is a $\mathscr{G}$-torsor over $A$.
\end{proof}

%% file: Infinitesimal_ramification.tex
\section{Ramification of degree $p$ purely inseparable extensions}\label{section:pur_insep_ramification}
We give the basic theory of degree $p$ purely inseparable extensions of field and DVRs.

\subsection{Purely inseparable field extensions of degree $p$}\label{section:pur_insep_field_ext_deg_p}
Let $K$ be a field extension of $k$. Then the purely inseparable extensions of $K$ of degree $p$ are precisely 
those of the form
        $$K_{F,a}=K[X]/(X^p-a), \quad a\notin K^p.$$
We have the $K$-linear derivation $\partial_{F,a}\in \Der_K(K_{F,a})$, sending $X$ to $1$, which satisfies
$\partial_{F,a}^{\circ p}=0$ and 
$\ker(\partial_{F,a})=K$. 
Two extensions $K_{F,a}$ and $K_{F,b}$ are $K$-isomorphic if and only if we can write $b=\sum_{i=0}^{p-1}u_i^pa^i$ for some $u_0,\dots,u_{p-1}\in K$ and $u_i\neq 0$ for at least one $i>0$.

\subsection{Purely inseparable DVR extensions of degree $p$}
Let $K$, $\upsilon$ and $(A,\fm_A,k_A)$ be as in \autoref{section:AS_ext_DVR}. Fix an element $a\in K$ that is not a $p$-th power, and let $B$ be the normalization of $A$ inside $K_{F,a}$. Then $B$ is a DVR with maximal ideal $\fm_B$, residue field $k_B$ and normalized valuation $\upsilon'$. Since $A$ is a DVR, $B$ is automatically a finite flat $A$-module, of degree $p=[K_{F,a}:K]$.

We define the \emph{ramification index} of $A\hookrightarrow B$ as follows:
        $$\widetilde{e}=\frac{p}{[k_B:k_A]}=\begin{cases}
        p & \text{if }\fm_AB=\fm_B^p, \\
        1 & \text{if }\fm_AB=\fm_B.
        \end{cases}$$
Note that if $\pi$ is a uniformizer of $A$, then $\upsilon'(\pi)=\widetilde{e}$.
We have the following analogue of \cite[III, \S 6, Lemme 3]{Serre_Corps_locaux}:

\begin{lemma}\label{lemma:explicit_p_basis}
Let $\varpi$ be a uniformizer of $B$, and $f\in B$ an element such that $k_B=k_A[\bar{f}]$. Then the elements
$\varpi^if^j$, for $0\leq i <\widetilde{e}$ and $0\leq j <p/\widetilde{e}$, form a basis of $B$ over $A$.
\end{lemma}
\begin{proof}
Since $B$ is flat of degree $p$ over $A$, and we are considering $p$ elements $\varpi^if^j$, it suffices to show that these elements generate $B$ as $A$-module. By Nakayama's lemma, it is sufficient to prove generation modulo $\fm_A$. We have $\fm_AB=\fm_B^{\widetilde{e}}$. If $i<\widetilde{e}$, then it is easy to check that generation modulo $\fm_B^i$ implies generation modulo $\fm_B^{i+1}$, and generation modulo $\fm_B$ holds by construction. Thus the statement follows.
\end{proof}

As in the Artin--Schreier case, to compute the ramification index we are allowed to replace $A$ by its strict henselization, or by the completion thereof. Let us introduce the following invariant:

\begin{definition}\label{def:inf_lambda}
For $g\in K^\mathrm{csh}$, we let
        $$\widetilde{\lambda}_{K,\upsilon}(g)=
        \max\{\upsilon^\mathrm{csh}(g+x^p)\mid x\in K^\mathrm{csh}\}\in \bZ\cup\{\infty\}.$$
We write interchangeably $\widetilde{\lambda}_{K,\upsilon}=
\widetilde{\lambda}_{A}=\widetilde{\lambda}$ if there is no risk of confusion. One sees easily that $\widetilde{\lambda}(g)$ is finite if and only if $g\notin (K^\mathrm{csh})^p$.
\end{definition}

This invariant can be computed using the following algorithm, similar to \autoref{algo:compute_lambda}: the proof is immediate.

\begin{algorithm}\label{algo:compute_inf_lambda}
Fix a uniformizer $\pi$ of $A$ and field of representatives $k_A\subset A$, so that $A^\mathrm{csh}\cong k_A^\mathrm{sep}\llbracket \pi \rrbracket$. Given $g\in K^\mathrm{csh}$, consider the following algorithm run in $K^\mathrm{csh}$:
    \begin{enumerate}
        \item Consider the following assertion:
            $$\upsilon^\mathrm{csh}(g)=n\in p\bZ\quad \text{and} \quad g_{n}\in (k_A^\mathrm{sep})^p.$$
        If this assertion about $g$ is false: return the integer $\upsilon^\mathrm{csh}(g)$. If this assertion about $g$ is true: go to the second step.

        \item If $\upsilon^\mathrm{csh}(g)=pm$ with $m\in \bZ$ and $g_{pm}=u^p\in (k_A^\mathrm{sep})^p$, replace $g$ by $g-u^p\pi^{pm}$ and go back to Step 1.
    \end{enumerate}
If $g\notin (K^\mathrm{csh})^p$ then this algorithm stops after finitely many steps and its output is $\widetilde{\lambda}_{K,\upsilon}(g)$.
\end{algorithm}

\begin{lemma}
In our situation, $\widetilde{\lambda}_{K,\upsilon}(a)\in \bZ$.
\end{lemma}
\begin{proof}
It suffices to show that $a\notin (K^\mathrm{csh})^p$. The proof is similar to the one of \autoref{lemma:lambda_well_defined}.
\end{proof}

\begin{lemma}\label{lemma:wild_iff_lambda_not_div_inf_case}
$\widetilde{\lambda}(a)$ is prime to $p$ if and only if $\widetilde{e}=p$.
\end{lemma}
\begin{proof}
We may assume that $A$ is strictly henselian and complete, and that $\upsilon(a)=\widetilde{\lambda}(a)$.

Suppose that $\upsilon(a)=s$ is prime to $p$. Then we can choose an uniformizer $\pi$ of $A$ such that $\pi^{s}=a$. Write $s=i+pr$ where $0<i<p$ and $r\in \bZ$. Then $K_{F,a}\cong K_{F,\pi^i}$ and so
        \begin{equation}\label{eqn:algebra_pres_wild_case}
        B\cong \big(
        k_A\llbracket \pi,x\rrbracket/(x^p-\pi^i)
        \big)^\nu = k_A\llbracket \pi^{1/p}\rrbracket
        \end{equation}
where we have used \cite[Proposition 2.2.1]{Posva_Resolution_of_1_foliations} for the last equality. So $\widetilde{e}=p$.

Conversely, suppose that $\widetilde{\lambda}(a)=pn$ with $n\in \bZ$. Then we can write $a=u\pi^{pn}$ for some uniformizer $\pi$ of $A$ and unit $u\in A^\times$. Then $K_{F,a}\cong K_{F,u}$. In particular, $u$ obtains a $p$-th root in $B$, and thus its residue $\bar{u}$ has a $p$-th root in $k_B$. If $k_B=k_A$, then there is $v\in A^\times$ such that $u=v^p+\pi f$ for some $f\in A$. But then
        $$a-(v\pi^n)^p=u\pi^{np}-(u-\pi f)\pi^{np}
        =\pi^{1+np}f$$
and we get $\upsilon(a)<\upsilon(a-(v\pi^n)^p)$, contradiction. So $k_A\neq k_B$, which means $\widetilde{e}=1$.
\end{proof}


\begin{remark}
Suppose that $a$ and $a’$ induce $K$-isomorphic extensions $K_{F, a}\cong K_{F, a’}$: then $a'=\sum_{i=0}^{p-1}u_i^pa^i$ as in \autoref{section:pur_insep_field_ext_deg_p}, and it is easily seen that $\widetilde{\lambda}(a)\neq \widetilde{\lambda}(a’)$ in general. Nevertheless, whether $\widetilde{\lambda}(a)$ is divisble by $p$ or not depends only on the $K$-isomorphism class of $K_{F,a}$, as shown by \autoref{lemma:wild_iff_lambda_not_div_inf_case} (or by a direct argument).
\end{remark}

The $A$-module structure of $B$ follows an alternative that is similar to the one we saw for Artin--Schreier extensions:

\begin{example}\label{example:structure_Frob_ext_wild_case}
Suppose that $\widetilde{\lambda}(a)$ is prime to $p$. We may assume that $\widetilde{\lambda}(a)=\upsilon(a)=s$.
Let $\pi\in A$ be a uniformizer and let $x\in \Frac(B)=K[X]/(X^p-a)$ be the element satisfying $x^p=a$. By \autoref{lemma:wild_iff_lambda_not_div_inf_case} we have $\upsilon'(\pi)=p$, so we deduce $\upsilon'(x)=s$. Let $c,d$ be integers satisfying $cs+dp=1$; note that $c$ is prime to $p$, and we can take it with $0<c<p$. 
Then
$\upsilon'(x^c\pi^d)=1$ and so $x^c\pi^d$ is a uniformizer of $B$. By \autoref{lemma:explicit_p_basis} we have
        $$B=\bigoplus_{j=0}^{p-1}A(x^c\pi^d)^j
        \quad \text{as }A\text{-module}.$$
If $\partial_{K,a}$ is the $K$-derivation of $K_{F,a}$ introduced before, we have 
        $$\upsilon'(\partial_{K,a}(x^c\pi^d))=
        \upsilon'(cx^{c-1}\pi^d)=s(c-1)+pd
        =1-s.$$
Letting $n=\max\{s-1,0\}$ we see that $\psi=\pi^n\partial_{K,a}\in \Der_A(B)$. We have $\psi^{\circ p}=0$, so $\psi$ defines an $\alpha_p$-action on $\Spec(B)$ 
with quotient $\Spec(A)$.

Replacing $A$ by $A^\mathrm{csh}$, \autoref{eqn:algebra_pres_wild_case} shows that we may assume $a=\pi$. Then $s=1$ above, which shows that the action of $\alpha_p$ defined by $\psi$ is free.
\end{example}

\begin{example}\label{example:structure_Frob_ext_fer_case}
Suppose that $\widetilde{\lambda}(a)$ is divisible by $p$.
We may assume that $\widetilde{\lambda}(a)=\upsilon(a)$. Choose an uniformizer $\pi$ of $A$, and write $a=u\pi^{pn}$ where $n$ is an integer and $u\in A^\times$ is a unit. Suppose that the residue $\bar{u}\in k_A^p$. Then we can write 
$u=v^p+x\pi $ for some $x\in A$ and unit $v\in A^\times$. Hence
        $$a-(v\pi^{n})^p=
        x\pi^{1+pn}$$
which shows that $\upsilon(a)<\upsilon(a+(v\pi^{n})^p)$, contradicting our choice of $a$. So $\bar{u}\notin k_A^p$.

There is an element $y\in K_{F,a}\cong K_{F,u}$ satisfying $y^p-u=0$. Since $u\in B^\times$ we get $y\in B$. In $k_B$ we have $\bar{y}^p-\bar{u}=0$; as $\bar{u}\notin k_A^p$ we obtain that $\bar{y}\notin k_A$. Therefore $B$ is generated as $A$-module by $\{y^j\mid 0\leq j\leq p-1\}$ \autoref{lemma:explicit_p_basis}. The derivation $\partial_{K,u}$ sends $y$ to $1$, so it (co)-restricts to an additive derivation of $B$. This derivation defines a free $\alpha_p$-action on $\Spec(B)$ with quotient $\Spec(A)$.
\end{example}

\begin{remark}
The previous two examples show that (after possibly an \'{e}tale cover of $A$) there is always a \emph{free} $\alpha_p$-action on $\Spec(B)$ with geometric quotient $\Spec(A)$; in other words, $\Spec(B)\to \Spec(A)$ is always a torsor for some $\alpha_{p,\Spec(A)}$-action. This contrasts with the Artin--Schreier case (\autoref{prop:torsor_AS_case}).
\end{remark}

Denote by $q\colon X=\Spec(B)\to \Spec(A)=Y$ the induced morphism on spectra. Let $\mathbf{x}_B=V(\fm_B)\subset X$ be the closed point, which we regard as an effective Cartier divisor. In order to parallel \autoref{eqn:ramification_AS_ext}, one may wonder if we can define a ``ramification break" $\widetilde{r}$ for $A\hookrightarrow B$ with the property that
        $$K_X=q^*K_Y+(p-1)\widetilde{r}\cdot \mathbf{x}_B.$$
This deserves some comments. The natural pullback map
$q^*\omega_Y\to \omega_X$ is zero since $q$ is purely inseparable, so we can only consider $\omega_X\otimes q^*\omega_Y^{-1}$ inside $\mathrm{Pic}(X)$, or equivalently $K_{X/Y}$ inside $\mathrm{Cl}(X)$.
Finite duality theory yields
$\omega_X\cong q^*\omega_Y\otimes q^!\sO_Y$. 
Alternatively, if $\sF$ is the $1$-foliation on $X$ defining $q$ and $K_\sF$ is the divisor class corresponding to the divisorial sheaf $(\det\sF)^{[-1]}$, then the adjunction formula reads
        $$K_X=q^*K_Y+(p-1)K_\sF \quad 
        \text{in }\mathrm{Cl}(X),$$
see \cite[Proposition 2.10]{Patakfalvi_Waldron}---or 
\cite[\S 2]{Rudakov_Shafarevich} for a different formulation. 

Trying to lift these equalities to the Weil divisor group of $X$ can only lead us astray. To wit, we saw in the previous examples that there always exists an $\alpha_p$-action on $X$, given by an additive derivation $\psi\in \Der_A(B)$, whose quotient is $q\colon X\to Y$. We have a section
        $$\xi_\psi\in q^!\sO_Y=\mathrm{Hom}_Y(q_*\sO_X,\sO_Y), 
        \quad \xi_\psi(b)=\psi^{\circ(p-1)}(b) 
        \text{ for all }b\in q_*\sO_X.$$
Tziolas shows in \cite[Theorem 8.1]{Tziolas_Quotient_by_alpha_or_mu} that the map
        $$\mathrm{Hom}_X(q^!\sO_Y,\sO_X)\longrightarrow \sO_X,
        \quad \mathfrak{v}\mapsto \mathfrak{v}(\xi_\psi)$$
is an isomorphism on its image $\mathcal{I}^{[1-p]}$, where $\mathcal{I}$ is the ideal sheaf of the scheme-theoretic closed locus of the $\alpha_p$-action, which is the ideal generated by $\psi(B)$. We have $\mathcal{I}=\sO_X(-r_\psi\mathbf{x}_B)$ for some integer $r_\psi\geq 0$ and so
        $$\omega_X\cong q^*\omega_Y\otimes \sO((p-1)r_\psi \cdot \mathbf{x}_B).$$
Now choose an uniformizer $\pi$ of $A$: for any integer $m>0$, the derivation $\pi^m\psi\in \Der_A(B)$ is still additive, hence its defines another $\alpha_p$-action on $B$ with sub-ring of invariants $A$. We have
$r_{\pi^m\psi}=m\widetilde{e}+r_\psi$ and so we see that the integer $r_\psi$ cannot be used in any reasonable way to define a ramification break.

%% file: Application_to_lc_families.tex
\section{Application to base-change 
of locally stable families}\label{section:application}
Let $(\mathbf{0}\in C)$ be a germ of $k$-curve, and let $f\colon (X,\Delta)\to C$ be a strongly locally stable family 
\cite[Definition 2.1.9]{Benozzo-Posva}. 
Let $h\colon (\mathbf{0}'\in C')\to C$ be a finite flat morphism where $(\mathbf{0}'\in C')$ is a germ of smooth $k$-curve, and the function field extension $k(C)\hookrightarrow k(C')$ is Artin--Schreier (i.e.\ Galois of order $p$). Let also $F^n\colon C\to C$ denote the $n$-th iterate $k$-linear Frobenius of $C$ (\footnote{
    We think of $F^n$ as the the morphism induced by the $k$-linear $n$-th Frobenius map $\sO_{C,\mathbf{0}}\to \sO_{C,\mathbf{0}}$, instead of the inclusion of $p^n$-th powers.
}). We consider the cartesian diagrams
    $$\begin{tikzcd}
    X'\arrow[d, "f'" left] \arrow[r, "h'"] & X \arrow[d, "f"] \\
    C' \arrow[r, "h"] & C
    \end{tikzcd}
    \quad \text{and} \quad
    \begin{tikzcd}
    X^{(n)}\arrow[d, "f^{(n)}" left] \arrow[r, "F^n_{X/C}"] & X \arrow[d, "f"] \\
    C \arrow[r, "F^n"] & C
    \end{tikzcd}$$
and we let $\Delta'=h'^*\Delta$ and $\Delta^{(n)}=F_{X/C}^{n,*}\Delta$. Then $X'$ and $X^{(n)}$ are normal, $f'\colon (X',\Delta')\to C'$ and $f^{(n)}\colon (X^{(n)},\Delta^{(n)})\to C$ are families of pairs, and $K_{X'}+\Delta'$ and $K_{X^{(n)}}+\Delta^{(n)}$ are $\bQ$-Cartier: see \cite[\S 3.2, \S 3.3]{Benozzo-Posva}.

\begin{theorem}\label{thm:main_application}
Assume that $f^{(n)}\colon (X^{(n)},\Delta^{(n)})\to C$ is strongly locally stable for every $n\geq 1$. Then $f'\colon (X',\Delta')\to C'$ is strongly locally stable.
\end{theorem}

The rest of the section is devoted to the proof of this theorem. We begin with some simple remarks. First of all, since the geometric fibers of $f'$ and of $f$ are the same, the content of the theorem is about the discrepancies of the pair $(X',\Delta'+X'_{\mathbf{0}'})$. Since log canonicity is preserved by \'{e}tale base-change \cite[Proposition 2.15]{Kollar_Singularities_of_the_minimal_model_program}, and the separable morphism $C'\to C$ is generically \'{e}tale, we reduce to the following situation:
$C'\to C$ is ramified at $\mathbf{0}'$, and we consider divisors over $X'$ that are vertical over $C'$. 

\begin{lemma}\label{lemma:divisor_from_downstairs}
Let $D$ be a prime divisor over $X'$. Then there exists a birational morphism $Y\to X$ from a normal variety such that if $\phi\colon Y'\to Y$ is the normalization of $Y$ in $k(X')$, then $D$ is a prime divisor on $Y'$. If the Galois $\bZ/p$-action on $Y'$ is not free at the generic point of $D$, then there exists a unique prime divisor $E\subset Y$ such that $\Supp(\phi^*E)=D$.
\end{lemma}
\begin{proof}
Since $k(X')/k(X)$ is Galois of degree $p$, we have an action of $G=\bZ/p$ on $k(X')$ with invariant sub-field $k(X)$. Given $\sigma$ in the image of $G\hookrightarrow \Aut_{k(X)}(k(X'))$ and a normalized discrete valuation $\upsilon$ of $k(X')$, the composition $\upsilon \circ \sigma$ is again a normalized discrete valuation. The prime divisor $D$ corresponds to a normalized discrete valuation $\upsilon_D$ of $k(X')$; by a theorem of Zariski and Abhyankar \cite[Lemma 2.22]{Kollar_Singularities_of_the_minimal_model_program}, there exists a birational morphism $Z'\to X'$ where $Z'$ is a normal quasi-projective variety over $k$, such that the centers on $Z'$ of the finitely many valuations $\{\upsilon_D\circ \sigma\mid \sigma \in G\}$ all have codimension one. In other words, $D$ is a prime divisor on $Z'$ and the Galois action on $k(Z')=k(X')$ extends to an open subset which contains the generic point of $D$. Since $Z'$ is quasi-projective over $k$, we can find an \emph{affine} open subset $Y'$ on which $G$ acts which contains the generic point of $D$ and its conjugates. Then we let $Y=Y'/G$: it is a normal affine variety endowed with a birational morphism $Y\to X$. The final statement follows immediately from \cite[\S X]{Raynaud_Anneaux_locaux_henseliens}, see the beginning of \autoref{section:AS_ext_DVR}.
\end{proof}

For the rest of the proof, let us fix a prime divisor $D$ over $X'$, which is vertical over $C'$. We let $Y\overset{\varphi}{\longrightarrow} X$ and $D\subset Y'$ be as in \autoref{lemma:divisor_from_downstairs}, with induced morphisms $\phi\colon Y'\to Y$ and $\psi\colon Y'\to X'$, and we write $E=\phi(D)$.



Since the residue field of $C$ at $\mathbf{0}$ is algebraically closed, the ramification of $\sO_{C,\mathbf{0}}\hookrightarrow \sO_{C',\mathbf{0}'}$ is necessarily wild. To prove \autoref{thm:main_application} we may replace $C$ by the henselization of $\sO_{C,\mathbf{0}}$; in particular, we can find a uniformizer $t$ of $\sO_{C,\mathbf{0}}$ and a positive integer $s>0$ prime to $p$ such that $C'\to C$ is the Artin--Schreier cover associated to the function $t^{-s}$. Clearly $\lambda_{\sO_{C,\mathbf{0}}}(t^{-s})=s+1$ (\autoref{def:ramification_break} and \autoref{prop:ramif_break_wild_case}).
Finally, we let 
        $$\lambda=\lambda_{\sO_{Y,E}}(t^{-s}), \quad
        \text{and} \quad m=\coeff_EY_\mathbf{0}.$$
In this situation, we have the following:

\begin{proposition}\label{prop:discrep_AS_divisor}
The extension of DVRs $\sO_{Y,E}\hookrightarrow \sO_{Y',D}$ is ramified. Moreover:
    \begin{enumerate}
        \item If the ramification of $\sO_{Y,E}\hookrightarrow \sO_{Y',E}$ is \emph{wild}, then
                $$a(D;X',\Delta'+X'_{\mathbf{0}'})
                =p\cdot a(E;X,\Delta+X_\mathbf{0})
                -(p-1)\cdot (sm+\lambda -1).$$
        \item If the ramification of $\sO_{Y,E}\hookrightarrow \sO_{Y',E}$ is \emph{ferocious}, then
                $$a(D;X',\Delta'+X'_{\mathbf{0}'})
                = a(E;X,\Delta+X_\mathbf{0})
                -\frac{p-1}{p}\cdot (sm+\lambda).$$
    \end{enumerate}
\end{proposition}
\begin{proof}
If $t^{-s}$ had an Artin--Schreier root in $\Frac(\sO_{Y,E})$, then $\Frac(\sO_{C,\mathbf{0}})$ would not be algebraically closed in $\Frac(\sO_{Y,E})$. However $f\colon X\to C$ satisfies the hypothesis of \autoref{lemma:affine_version_of_Stein}, cf.\ \cite[Definition 2.1.9]{Benozzo-Posva}, and $Y\to X$ is birational, so we would obtain a contradiction. 
Moreover, $t^{-s}$ does not belong to $\sO_{Y,E}$, as seen for example
by extension-contraction of ideals along the faithfully flat extension $\sO_{C,\mathbf{0}}\hookrightarrow \sO_{Y,E}$.
So $\sO_{Y,E}\hookrightarrow \sO_{Y',D}$ is indeed ramified.

To prove the discrepancy formulas, we write $a_E=a(E;X,\Delta+X_\mathbf{0})$ and $a_D=a(D;X',\Delta'+X'_{\mathbf{0}'})$ for ease of notation. We let $e$ and $r$ be respectively the ramification index and the ramification break of $\sO_{Y,E}\hookrightarrow \sO_{Y',E}$.
Consider the following commutative diagram
        $$\begin{tikzcd}
        Y \arrow[d, "\varphi"] & Y\times_CC'\arrow[d, "\varphi'"] 
        \arrow[l, "g" above]
        & Y' \arrow[l, "\nu" above] \arrow[ll, bend right, "\phi" above] \arrow[dl, "\psi"]
        \\
        X & X'. \arrow[l,"h'" above]
        \end{tikzcd}$$
Since the ramification of $C'\to C$ is wild, we have $h^*[\mathbf{0}]=p\cdot [\mathbf{0}']$ and $K_{C'}=h^*K_C+(p-1)(s+1)$, see \autoref{eqn:ramification_AS_ext} and \autoref{prop:ramif_break_wild_case}. Therefore
$g^*Y_\mathbf{0}=p\cdot (Y\times_CC')_{\mathbf{0}'}$, and
$K_{X'}=h'^*K_X+(p-1)(s+1)X'_{\mathbf{0}'}$ by \cite[0C17]{Stacks_Project}. 
Thus we can write, around the generic point of $D$:
    \begin{eqnarray*}
        \psi^*h'^*(K_X+\Delta+X_\mathbf{0}) &=&
        \psi^*\big((K_{X'}+\Delta+X'_{\mathbf{0}'})+
        (-1-(p-1)(s+1)X'_{\mathbf{0}'}\big) \\
        &=&
        K_{Y'}-a_D\cdot D+(1-p)\cdot \nu^*(Y\times_CC')_{\mathbf{0}'} \\
        &=&
        K_{Y'}-a_D\cdot D+s\frac{1-p}{p}m\cdot \phi^*E \\
        &=&
        K_{Y'}+\left(
        \frac{1-p}{p}sme-a_D
        \right)\cdot D.
    \end{eqnarray*}
On the other hand, still around the generic point of $D$:
    \begin{eqnarray*}
        \phi^*\varphi^*(K_X+\Delta+X_\mathbf{0}) &=& 
        \phi^*(K_Y-a_E\cdot E) \\
        &=&
        K_{Y'}-\big(
        (p-1)r+ea_E\big) \cdot E.
    \end{eqnarray*}
By comparing the two expressions, we find
        $$a_D=e\cdot a_E+(p-1)\left(
        r-\frac{sme}{p}
        \right).$$
If the extension is wild, then $e=p$ and $r=-\lambda+1$; while if the extension is ferocious, $e=1$ and $r=-\lambda/p$ (\autoref{prop:ramif_break_wild_case} and \autoref{prop:ramif_break_feroce_case}). Substituting into the above expression yields the result.
\end{proof}

Recall that the prime divisor
$E\subset Y$ is called a \emph{wild divisor} over $C$ if $\coeff_EY_\mathbf{0}\in p\bZ_{>0}$, and a \emph{tame divisor} otherwise \cite[\S 4]{Benozzo-Posva} (\footnote{
    This terminology is unrelated to tame and wild extensions of DVRs.
}). Tame divisors are simple to handle, as the following result shows:

\begin{corollary}\label{cor:AS_discrep_tame_div}
Assume that $E$ is a tame divisor for $Y\to C$. Then $\sO_{Y,E}\hookrightarrow \sO_{Y',D}$ is wildly ramified, and
        $$a(D;X',\Delta'+X'_{\mathbf{0}'})+1
                =p\cdot \big(a(E;X,\Delta+X_\mathbf{0})+1\big).$$
\end{corollary} 
\begin{proof}
By assumption $m$ is prime to $p$. Passing to the strict henselization of $\sO_{Y,E}$ if necessary, we may find a uniformizer $\pi$ such that $t=\pi^m$. Then 
$\sO_{Y,E}\hookrightarrow \sO_{Y',D}$ is the Artin--Schreier extension defined by the function
$t^{-s}=\pi^{-ms}$. Hence $\lambda=-ms$, which is prime to $p$, so the ramification is wild by \autoref{eqn:lambda_and_ram_type}, and \autoref{prop:discrep_AS_divisor} yields
        $$a(D;X',\Delta'+X'_{\mathbf{0}'})
        =p \cdot a(E;X,\Delta+X_{\mathbf{0}})+p-1$$
as claimed.
\end{proof}

For every $n\geq 1$, we let $Y^{(n)}$ be the normalization of $Y\times_{C,F^n}C$; it supports a unique prime divisor $E^{(n)}$ which dominates $E$ through the structural morphism $Y^{(n)}\to Y$. We define the divisors $D^{(n)}$ over $X'^{(n)}$ similarly. We recall from \cite{Benozzo-Posva} that:

\begin{proposition}[Frobenius descent]\label{prop:Frob_descent}
If $a(D^{(n)};X'^{(n)},X'^{(n)}_{\mathbf{0}'}+\Delta'^{(n)})\geq -1$, then $a(D;X,X_{\mathbf{0}'}+\Delta')\geq -1$.
\end{proposition}
\begin{proof}
This is proved in \cite[Proposition 6.1.1 and proof of Proposition 6.1.2]{Benozzo-Posva}.
\end{proof}

\begin{corollary}\label{cor:AS_discrep_event_tame_div}
Suppose that $E^{(n)}$ is tame for $n\gg 1$. Then $a(D;X',X'_{\mathbf{0}'}+\Delta')\geq -1$.
\end{corollary}
\begin{proof}
We claim that the diagram
    \begin{equation}\label{eqn:AS_commutes_with_Frob}
    \begin{tikzcd}
    C' \arrow[d, "h" left] & C' \arrow[d, "h"] \arrow[l, "F^n" above] \\
    C & C \arrow[l, "F^n" above]
    \end{tikzcd}
    \end{equation}
is commutative. This is equivalent to the following field-theoretic statement: the extension $k(C)\hookrightarrow L=k(C)[u,v]/(u^p-t,v^p+v-t^{-s})$ can be written in two ways:
        $$\left(k(C)_{F,t}\right)_{\wp,t^{-s}}=L
        = \left( k(C)_{\wp,t^{-s}}\right)_{F,t},$$
which is clear. So $D^{(n)}$ is the unique prime divisor over $X'^{(n)}$ dominating $E^{(n)}$. By \autoref{cor:AS_discrep_tame_div} we deduce that $a(D^{(n)};X'^{(n)},X'^{(n)}_{\mathbf{0}'}+\Delta'^{(n)})\geq -1$, and the statement follows from \autoref{prop:Frob_descent}.
\end{proof}

\begin{remark}
Assume that $E$ is tame. Then it follows from \autoref{cor:AS_discrep_tame_div} and \cite[Proposition 7.3.8]{Benozzo-Posva} that
    $$a(D;X',\Delta'+X'_{\mathbf{0}'})
    =a(E^{(1)};X^{(1)},\Delta^{(1)}+X^{(1)}_\mathbf{0}),$$
as both sides are equal to $p\cdot a(E;X,\Delta+X_\mathbf{0})+(p-1)$.
\end{remark}

The remaining case, i.e.\ when $E^{(n)}$ is not tame for any $n\geq 1$, is more subtle to analyze. Write $m_n=\coeff_{E^{(n)}}Y^{(n)}_\mathbf{0}$. For every $n\geq 1$ we have a commutative diagram
        $$\begin{tikzcd}
        Y^{(n-1)} \arrow[d] & Y^{(n)} \arrow[d]\arrow[l, "q_n" above] \\
        C & C \arrow[l, "F" above]
        \end{tikzcd}$$
We have $F^*[\mathbf{0}]=p\cdot [\mathbf{0}]$, and so $q_n^*Y^{(n-1)}_\mathbf{0}=pY^{(n)}_\mathbf{0}$.
Moreover, by \cite[Lemma 5.0.7]{Benozzo-Posva} we have $q_n^*E^{(n-1)}=E^{(n)}$ if
and only if $E^{(n)}$ is non-invariant for the $1$-foliation
$T_{Y^{(n-1)}/C}$. Thus we obtain that:
        \begin{equation}\label{eqn:how_coeff_changes}
        m_n=\begin{cases}
        \frac{m_{n-1}}{p} & \text{if }E^{(n-1)}\text{ is non-}T_{Y^{(n-1)}/C}\text{-invariant}, \\
        m_{n-1} & \text{otherwise.}
        \end{cases}
        \end{equation}
As the $m_n$s are integers, we see that the sequence is eventually constant. It may happen that the limit value is divisible by $p$, as shown by the following example, which was communicated to me by Akiyoshi Sannai and found using the AI tool he is developing: 

\begin{example}\label{example:stubborn_divisor}
Consider the locally stable family $X=\bA^2_{x,t}\to \bA^1_t$, where the morphism is the projection onto the second coordinate. The morphism
        $$Y=\bA^2_{x,y}\longrightarrow X, 
        \quad (x,y)\mapsto (x,x^p(1+x+yx^2))$$
is birational, and the prime divisor $E=V(y)\subset Y$ has multiplicity $p$ inside the fiber over the origin $\mathbf{0}\in \bA^1_t$. So $E$ is wild over $C$. We claim that if $p>2$, then for any $n\geq 1$ the multiplicity of $E^{(n)}$ in $Y^{(n)}_\mathbf{0}$ is equal to $p$. 

To see this, we start by finding a special uniformizer of $\widehat{\sO}_{Y,E}$. Consider the polynomial
        $H(T)=T^p(1+xT)-(1+x+yx^2)\in \widehat{\sO}_{Y,E}[T]$. 
An easy computation shows that $H(1+yx)$ has valuation $p$ in $\sO_{Y,E}$, while $H'(1+yx)$ has valuation $1$. So, as soon as $p>2$, Hensel's lemma in the form \cite[III, \S 4, n.5, Corollaire 1]{Bourbaki_AC_I-IV} shows that there exists a unique $v\in \widehat{\sO}_{Y,E}$, congruent to $1+yx$ modulo the square of the maximal ideal, such that $H(v)=0$. So $\pi=vx$ is a uniformizer of $\widehat{\sO}_{Y,E}$, and by construction
        $$t=x^p(1+x+yx^2)=x^pv^p(1+xv)=\pi^p(1+\pi).$$
Clearly $\widetilde{\lambda}(t)=-p+1$ is prime to $p$, and the proof of \autoref{lemma:wild_iff_lambda_not_div_inf_case} shows that $\pi_1=\pi^{1/p}$ is a uniformizer of $\sO_{Y^{(1)},E^{(1)}}$ (which can also be checked directly). So $E^{(1)}$ has multiplicity $p$ in $Y^{(1)}_\mathbf{0}$. We have $t^{1/p}=u_1\pi_1^p$ for some unit $u_1$; after comparing to $t=(1+\pi)\pi^p$ we obtain $u_1=1+\pi_1$. So, iterating the argument, we obtain that for any $n\geq 1$, the morphism       
    $$\widehat{\sO}_{C,\mathbf{0}}\longrightarrow 
    \widehat{\sO}_{Y^{(n)},E^{(n)}}$$
is given by $t\mapsto (1+\pi_n)\pi_n^p$ where $\pi_n=\pi^{1/p^n}$.
\end{example}

By Frobenius descent, the study of discrepancies of non-eventually-tame divisors is reduced to the study of the discrepancy of divisors of the following type:

\begin{definition}[Stubborn divisors]\label{def:stubborn_div}
We say that a prime divisor $E\subset Y$ is \emph{stubborn} over $C$ if: it is wild, and
$\coeff_{E}Y_\mathbf{0}=\coeff_{E^{(n)}}Y^{(n)}_\mathbf{0}$ for every $n\geq 1$ (\footnote{
    If $E$ is tame then it is also true that $\coeff_{E}Y_\mathbf{0}=\coeff_{E^{(n)}}Y^{(n)}_\mathbf{0}$for every $n\geq 1$, but we reserve the name stubborn for wild divisors.
}).
\end{definition}

\begin{remark}\label{rmk:stubbon_is_inv}
From what we have said above, we see that $E$ is a stubborn divisor, if and only if $E^{(n)}$ is $T_{Y^{(n)}/C}$-invariant for every $n$. 
\end{remark}

\begin{lemma}\label{lemma:nice_coord_for_stubborn}
Suppose that $E\subset Y$ is stubborn. Then for every $m\geq 1$, there exists a uniformizer $\pi\in \sO_{Y,E}$ such that 
$\sO_{Y^{(m)},E^{(m)}}=\sO_{Y,E}[\pi^{1/p^m}]$.
\end{lemma}
\begin{proof}
We write $A=\sO_{Y,E}$ and $B=\sO_{Y^{(m)},E^{(m)}}$ for simplicity; the maximal ideals are denoted by $\fm_{(-)}$, the residue fields by $k_{(-)}$. The inclusion of $A$ inside $A^{1/p^m}$ factors through $B$. Since $E$ is stubborn, each successive extension
        $$A\hookrightarrow \sO_{Y^{(1)},E^{(1)}}
        \hookrightarrow \dots \hookrightarrow B$$
has ramification index $p$ and induces a trivial extension of residue fields. So $\fm_AB=\fm_B^{p^m}$ and $k_A=k_B$. As $B$ and $A$ are DVRs, we deduce that $B$ is flat of degree $p^m$ as $A$-module.
If $\xi$ is a uniformizer of $B$, then arguing as in \autoref{lemma:explicit_p_basis} we see that $\{\xi\}$ is a $p^m$-basis of $B$ over $A$ (which means that $\{1,\xi,\dots,\xi^{p^m-1}\}$ is a basis of the $A$-module $B$). If $\{\bar{z}_1,\dots,\bar{z}_n\}$ is a $p^m$-basis basis of the residue field $k_{A^{1/p^m}}$ over $k_B=k_A$, then for any choice of lifts $z_1,\dots,z_n\in A^{1/p^m}$ we see that $\{\xi,z_1,\dots,z_n\}$ is a $p^m$-basis of $A^{1/p^m}$ over $A$. Since $\xi$ belongs to the maximal ideal of $B$, it also belongs to the maximal ideal of $A^{1/p^m}$: therefore it must be a uniformizer of $A^{1/p^m}$, and $\xi^{p^m}$ is a uniformizer of $A$.
\end{proof}

Note that the uniformizer $\pi$ in the statement depends on $m$ in general:
unlike what happens in \autoref{example:stubborn_divisor}, there might not exist a uniformizer $\pi$ of $\widehat{\sO}_{Y,E}$ such that $\widehat{\sO}_{Y^{(m)},E^{(m)}}=\widehat{\sO}_{Y,E}[\pi^{1/p^m}]$ for every $m\geq 1$.

\medskip
Our next step is to relate the discrepancy of a stubborn divisor $E$ to its \emph{wild coefficient} $\gamma=\gamma_w(E)$, defined in \cite[Definition 4.0.1]{Benozzo-Posva}. To achieve this, the main step is the following proposition, which uses the some results from the theory of $n$-foliations. For the latter, we refer the reader to \cite{Grabowski_Power_towers} or to the upcoming \cite{foliations_in_progress}.

\begin{proposition}\label{prop:discrep_stubborn_div}
Suppose that $E$ is a stubborn divisor. Then, after possibly replacing $E$ by $E^{(m)}$ for some $m\geq 1$ in case $\Delta\neq 0$, we have
        $$a(E^{(n)};X^{(n)},\Delta^{(n)}+X_\mathbf{0}^{(n)})
        =p^n\cdot ( a(E;X,\Delta+X_\mathbf{0})-\gamma) +\gamma$$
for all $n\geq 1$.
\end{proposition}
\begin{proof}
We have a commutative diagram
        $$\begin{tikzcd}
        Y^{(n)} \arrow[r, "q_n"]\arrow[d, "\varphi^{(n)}"] & Y \arrow[r, "r_n"] \arrow[d, "\varphi"] &
        Y^{(n)} \arrow[d, "\varphi^{(n)}"] \\
        X^{(n)} \arrow[r, "\bar{q}_n"] \arrow[d] & X \arrow [r, "\bar{r}_n"] \arrow[d] & X^{(n)} \arrow[dl] \\
        C \arrow[r, "F^n"] & C
        \end{tikzcd}$$
where $r_n$ and $\bar{r}_n$ are the quotients by the $n$-foliations $\Diff_{Y/C}^{<p^n}$ and $\Diff_{X/C}^{< p^n}$ respectively, and the compositions $r_n\circ q_n$ and $\bar{r}_n\circ \bar{q}_n$ are the $n$-th $k$-linear Frobenius. 

Since $E$ is stubborn we have $q_n^*E=p^nE^{(n)}$ for every $n$. As $r_n\circ q_n$ is the $n$-th $k$-linear Frobenius, we have $r_n^*E^{(n)}=E$ for every $n$.

Recall that $\Delta^{(n)}=q_n^*\Delta$ by definition. Since $(X^{(n)},\Delta^{(n)})\to C$ is locally stable, the coefficients of $\Delta^{(n)}$ stay bounded by $1$ for all $n\geq 1$. As $\bar{q}_n^*\bar{r}_n^*\Delta^{(n)}=p^n\Delta^{(n)}$, we can replace $(X,\Delta)\to C$ by $(X^{(m)},\Delta^{(m)})\to C$ for $m$ large enough, so that $\bar{r}_n^*\Delta^{(n)}=p^n\Delta$ for every $n$.

Since $X\to C$ is smooth in codimension one, $\Diff_{X/C}^{(n)}$ is an \emph{Ekedahl $n$-foliation} \cite[Proposition 5.66 and Definition 5.47]{Grabowski_Power_towers} (or a \emph{Cartier $n$-foliation} in the terminology of \cite{foliations_in_progress}). 

We claim that $\Diff_{Y/C}^{<p^n}$ is also an Ekedahl $n$-foliation in a neighbourhood of $E$. Indeed, by \autoref{lemma:nice_coord_for_stubborn}, we can find a $p$-basis $\pi,x_1,\dots,x_n$ of $\sO_{Y,E}$, where $\pi$ is a uniformizer, such that the base-change
$q_n\colon Y^{(n)}\to Y$ expresses $\sO_{Y^{(n)},E^{(n)}}$ as $\sO_{Y,E}[\pi^{1/p^n}]$. Dually, the $C$-linear Frobenius $r_n\colon Y\to Y^{(n)}$ is the quotient in a neighbourhood of the generic point of $E$ by the $n$-foliation generated by the differential operators $(1/\mathbf{i}!)(\partial^\mathbf{i}/\partial \mathbf{x}^\mathbf{i})$ for $|\mathbf{i}|< p^n$ (see \cite[\S 16.11]{EGA_IV.4} for details about these operators). It is easily seen that $\mathrm{gr}\Diff_{Y/C}^{<p^n}$ is saturated in $\mathrm{gr}\Diff_{Y/k}^{<p^n}$ at the generic point of $E$, and thus $\Diff_{Y/C}^{<p^n}$ is Ekedahl there \cite[Definition 5.45]{Grabowski_Power_towers}.

Consider the $1$-foliations $\sF=T_{Y/C}$ and $\overline{\sF}=T_{X/C}$. Observe that $K_{\overline{\sF}}+\Delta$ is $\bQ$-Cartier: for $f\colon X\to C$ is smooth on a big open subset, over which the sequence
    $$0\to T_{X/C}\to T_X\to f^*T_C\to 0$$
is exact,
so we have an equality of $\bQ$-divisor classes
    $K_{\overline{\sF}}=K_X-f^*K_C$.
Since $K_X+\Delta$ and $f^*K_C$ are $\bQ$-Cartier, so is $K_{\overline{\sF}}+\Delta$.

To simplify notation, we write
$a_0=a(E;X,\Delta+X_\mathbf{0})$
and 
$a_n=a(E^{(n)};X^{(n)},\Delta^{(n)}+X^{(n)}_\mathbf{0})$.
According to the adjunction formula of \cite[Theorem 7.1, Proposition 7.2]{Grabowski_Power_towers}, in a neighbourhood of the generic point of $E$ we have
    \begin{eqnarray*}
        \varphi^*\bar{r}_n^*(K_{X^{(n)}}+\Delta^{(n)}+X^{(n)}_\mathbf{0})
        &=&
        \varphi^*(K_X+\Delta+X_\mathbf{0}+(p^n-1)(K_{\overline{\sF}}+\Delta)) \\
        &=&
        K_Y-a_0 E+\varphi^{-1}_*\Delta+(p^n-1)\varphi^*(K_{\overline{\sF}}+\Delta),
    \end{eqnarray*}
and 
    \begin{eqnarray*}
        r_n^*\varphi^{(n),*}(K_{X^{(n)}}+\Delta^{(n)}+X_\mathbf{0}^{(n)})
        &=&
        r_n^*\left(K_{Y^{(n)}}-a_n E^{(n)}
        +\varphi^{(n),-1}_*\Delta^{(n)}\right) \\
        &=&
        K_Y-a_n E+\varphi^{-1}_*\Delta
        +(p^n-1)(K_\sF+\varphi^{-1}_*\Delta).
    \end{eqnarray*}
Comparing the two expressions yields:
    $$a_n=a_0+(p^n-1)a(E;\overline{\sF},\Delta).$$
It follows from \cite[Remark 7.3.14]{Benozzo-Posva} (after an immediate adaptation of \cite[Proposition 4.0.3]{Benozzo-Posva} to the logarithmic setting) that 
$a(E;\overline{\sF},\Delta)=a(E;X,\Delta+X_\mathbf{0})-\gamma_w(E)$. The desired formula follows at once.
\end{proof}

The importance of the previous proposition is the following immediate corollary:

\begin{corollary}\label{cor:discrep_bigger_than_wild_coeff}
If $E$ is a stubborn divisor and $a(E^{(n)};X^{(n)},\Delta^{(n)}+X^{(n)}_\mathbf{0})\geq -1$ for every $n\geq 1$, then after replacing $(X,\Delta)\to C$ by $(X^{(m)},\Delta^{(m)})\to C$ for some $m$ large enough, it holds that $a(E;X,\Delta+X_\mathbf{0})\geq \gamma$. \hfill \qedsymbol
\end{corollary}

\begin{remark}
In fact, the formula of \autoref{prop:discrep_stubborn_div} shows that as $n$ grows, the discrepancies of $E^{(n)}$ either go to $-\infty$, or to $+\infty$, or eventually stabilize to a non-negative integer.
\end{remark}

\begin{remark}
The proof of \autoref{prop:discrep_stubborn_div} shows that, in the terminology of the upcoming \cite{foliations_in_progress}, the $\infty$-foliation on any birational model $Y$ of $X$ induced by the morphism $Y\to C$, is ``eventually Cartier". The authors \cite{foliations_in_progress} kindly informed me that they have independently obtained another proof of this statement in their study of foliations induced by flat dominant separable morphisms.
\end{remark}

Next, we relate the discrepancy of $D$ to the discrepancy of $E^{(1)}$ (and for this we do not need to assume that $E$ is stubborn). We write
        $$\widetilde{\lambda}=
        \widetilde{\lambda}_{\sO_{Y,E}}(t^{-s}) 
        \quad \text{and} \quad 
        \gamma=\gamma_w(E)=\text{ wild coefficient of }
        E \text{ over }C.$$
Let us introduce the quantities 
        \begin{eqnarray*}
         \alpha_\mathrm{w}(E)&=& p\cdot a_E
                -(p-1)\cdot (sm+\lambda -1), \\ 
        \alpha_\mathrm{f}(E)&=& a_E
                -\frac{p-1}{p}\cdot (sm+\lambda).
        \end{eqnarray*}
where we write $a_E=a(E;X,\Delta+X_\mathbf{0})$ for simplicity.
One of these two computes the discrepancy of $D$ over $(X',\Delta'+X'_{\mathbf{0}'})$ by \autoref{prop:discrep_AS_divisor}.
Observe that $\alpha_\mathrm{w}(E)+1=p\cdot (\alpha_\mathrm{f}(E)+1)$, and thus:

\begin{lemma}\label{lemma:one_implies_the_other}
$\alpha_\mathrm{w}(E)\geq -1$ if and only if $\alpha_\mathrm{f}(E)\geq -1$. \hfill \qedsymbol
\end{lemma}

Moreover, we have:

\begin{lemma}\label{lemma:comparison_lambda_gamma}
Assume that $E$ is $T_{Y/C}$-invariant. Then $\widetilde{\lambda}$ is prime to $p$, $\gamma=\widetilde{\lambda}+ms-1$ and   
    $$a(E^{(1)};X^{(1)},\Delta^{(1)}+X_\mathbf{0}^{(1)})
        =
        p\cdot a_E-(p-1)\gamma.$$
\end{lemma}
\begin{proof}
That $\widetilde{\lambda}$ is prime to $p$ follows from $q_1^*E=pE^{(1)}$ and \autoref{lemma:wild_iff_lambda_not_div_inf_case}. 
The formula for the discrepancy of $E^{(1)}$ follows from \cite[Proposition 7.3.3]{Benozzo-Posva}.

To prove the formula relating $\gamma$ and $\widetilde{\lambda}$, choose a uniformizer $\pi$ of $\sO_{Y,E}^\mathrm{csh}$ and write
$t=u\pi^m$ with $u\in \sO_{Y,E}^{\mathrm{csh},\times}$. Since $s$ is prime to $p$, we have $d(u^{-s})=-su^{-s-1}du$ and so
$\gamma=\ord_Ed(u^{-s})$. Let us expand 
        $$u^{-s}=\sum_{i\geq 0}u_i\pi^i, \quad u_i\in 
       \mathfrak{K}= k(E)^\mathrm{sep},$$
and let $j$ the smallest index such that: either $j\notin p\bZ$ and $u_j\neq 0$, or $j\in p\bZ$ and $u_j\notin \mathfrak{K}^p$. Since
    $$t^{-s}=\sum_{i\geq 0}u_i\pi^{i-ms}$$
and as $m$ is divisible by $p$, it follows from \autoref{def:inf_lambda} that $\widetilde{\lambda}=j-ms$. As $\widetilde{\lambda}$ is prime to $p$, so is $j$. Hence $d(u^{-s})\equiv ju_j\pi^{j-1}$ modulo $\pi^j\Omega^1_{\sO_{Y,E}^\mathrm{csh}}$, and we obtain $\gamma=j-1=\widetilde{\lambda}+ms-1$ as desired.
\end{proof}

\begin{lemma}\label{lemma:when_lambda_bigger_inf_lambda}
Assume that $\lambda>\widetilde{\lambda}$. Then
$0>\widetilde{\lambda}\geq -\frac{m}{p}s$ and
$0<\lambda-\widetilde{\lambda}< \frac{m}{p}s$.
\end{lemma}
\begin{proof}
Fix an isomorphism $\sO_{Y,E}^\mathrm{csh}\cong k(E)^\mathrm{sep}\llbracket \pi\rrbracket$, so that $t=u\pi^m$ for a unit $u$. We expand
        $$t^{-s}=u^{-s}\pi^{-ms}
        =\sum_{i\geq -ms}v_i\pi^i, \quad 
        v_i\in k(E)^\mathrm{sep}.$$
We run \autoref{algo:compute_lambda} and \autoref{algo:compute_inf_lambda} on the series $\sum_i v_i\pi^i$ to compute respectively $\lambda$ and $\widetilde{\lambda}$ respectively. The second algorithm, which computes $\widetilde{\lambda}$, is simple: we remove the leading term of the series until it is not a $p$-th power. The first algorithm is slightly different: as long as the leading term is a $p$-th power, we remove it and add its $p$-th root to the series. In both cases, when the algorithm stops the output is the valuation of the modified series.  
From these descriptions, the important observation is that when going through the terms
        $$v_{-ms}\pi^{-ms},\dots,v_{-ns-1}\pi^{-ns-1}
        \quad \text{with }m=pn,$$
one algorithm stops if and only if the other algorithm stops, and if so they stop at the same step. Now we assume $\lambda>\widetilde{\lambda}$, 
so in our situation none of the two algorithms stops before considering a series of valuation $\geq -ns$. Moreover, recall that $0>\lambda$ by \autoref{lemma:lambda_well_defined}. Therefore,
    $$0>\lambda >\widetilde{\lambda}\geq -ns$$
and both statements follow.
\end{proof}

We can finally prove the main result of this section.

\begin{proof}[Proof \autoref{thm:main_application}]
Fix a divisor $E$ over $X$ which is vertical over $C$, and let $D$ be the unique prime over $X'$ which dominates $E$. Our goal is to prove the inequality $a(D;X',\Delta'+X'_{\mathbf{0}'})\geq -1$.
By assumption $(X^{(n)},\Delta^{(n)})\to C$ is locally stable for every $n\geq 1$, so in particular $a(E^{(n)};X^{(n)},\Delta^{(n)}+X^{(n)}_\mathbf{0})\geq -1$ for every $n$.
If $E^{(n)}$ is tame for $n\gg 1$, then we are done
by \autoref{cor:AS_discrep_event_tame_div}. So we may assume that $E^{(n)}$ is stubborn for every $n$ large enough. By Frobenius descent for $D$ (\autoref{prop:Frob_descent}) and the commutative diagram \autoref{eqn:AS_commutes_with_Frob}, we may replace $(X,\Delta)\to C$ by $(X^{(n')},\Delta^{(n')})\to C$ for some $n'\gg n$, and assume that $E$ is stubborn and that \autoref{cor:discrep_bigger_than_wild_coeff} holds with $m=0$. By \autoref{prop:discrep_AS_divisor} and \autoref{lemma:one_implies_the_other}, it suffices to show that $\alpha_\mathrm{w}(E)\geq -1$. We write
    $$\lambda=\lambda_{\sO_{Y,E}}(t^{-s}), 
    \quad
    \widetilde{\lambda}=\widetilde{\lambda}_{\sO_{Y,E}}(t^{-s})
    \quad \text{and}\quad
    \gamma=\gamma_w(E)$$
as before. If $\widetilde{\lambda}\geq \lambda$ then using \autoref{lemma:comparison_lambda_gamma} we have:
    \begin{eqnarray*}
        -1\leq a(E^{(1)};X^{(1)},\Delta^{(1)}+X^{(1)}_\mathbf{0}) &=&
        p\cdot a(E;X,\Delta+X_\mathbf{0})-(p-1)\gamma \\
        &=&
        p\cdot a(E;X,\Delta+X_\mathbf{0})-(p-1)\cdot (\widetilde{\lambda}+ms-1) \\
        &\leq &
        p\cdot a(E;X,\Delta+X_\mathbf{0})-(p-1)\cdot (\lambda+ms-1)\\
        &=&
        \alpha_\mathrm{w}(E).
    \end{eqnarray*}
The remaining case is when $\widetilde{\lambda}< \lambda$. Write $m=pn$. Using \autoref{cor:discrep_bigger_than_wild_coeff}, \autoref{lemma:comparison_lambda_gamma} and \autoref{lemma:when_lambda_bigger_inf_lambda} we have
    \begin{eqnarray*}
    \alpha_\mathrm{w}(E) &\geq &
    p \gamma -(p-1)(sm+\lambda-1) \\
    &=&
    ms+(p-1)(\widetilde{\lambda}-\lambda)
    +\widetilde{\lambda}-1\\
    &> &
    pns-(p-1)ns-ns-1 \\
    &=& -1.
    \end{eqnarray*}
This completes the proof.
\end{proof}

\begin{remark}[Hu--Zong's proof in \cite{Hu_Zong}]\label{rmk:comp_with_HZ}
The first steps of our argument and that of \cite{Hu_Zong} are parallel, although the proofs are different. The formulas in \autoref{prop:discrep_AS_divisor} to compute the discrepancy of $D$, are equivalent to the formulas in \cite[Proposition 3.1]{Hu_Zong}. The tame case (\autoref{cor:AS_discrep_tame_div}) corresponds to \cite[\S 3.1]{Hu_Zong}, and we obtain the same discrepancy formula.
The cases covered by \cite[\S 3.2.1]{Hu_Zong} correspond to some eventually-tame divisors. The remaining of the two arguments are more complicated to relate.

The main conceptual difference between our method and the one in \cite{Hu_Zong} is that the latter focuses on approximating the presentation of $\sO_{Y',D}$ as $\sO_{Y,E}$-algebra, starting from the base-change equation of $Y\times_{C}C'$, blowing-up to get partial normalizations, and taking Frobenius base-change to deal with leading terms that are $p$-th powers.
Instead, we argue in terms of ``module-theoretic" ramification data on $Y$ (i.e.\ the invariants $\lambda$, $e$, $r$, $\widetilde{\lambda}$, $\widetilde{e}$ and $\gamma$).

We note that Frobenius descent of local stability is not mentioned in \cite{Hu_Zong}. Moreover, we observe that \cite[Claim 3.3]{Hu_Zong} cannot hold for the local ring $\widehat{\sO}_{Y,E}$ of \autoref{example:stubborn_divisor} if we stick to a set of local coordinates containing the ``special" uniformizer $\pi$, thus raising the question of the role of coordinate change in \emph{op.cit.}
\end{remark}